\documentclass{amsart}

\usepackage{graphicx}
\usepackage{amsmath, amssymb, amsthm}
\usepackage{epstopdf}
\usepackage{hyperref}

\newtheorem{theorem}{Theorem}[section]
\newtheorem{lemma}[theorem]{Lemma}
\newtheorem{proposition}[theorem]{Proposition}
\theoremstyle{definition}
\newtheorem{definition}[theorem]{Definition}
\newtheorem{example}[theorem]{Example}

\newtheorem{corollary}[theorem]{Corollary}

\theoremstyle{remark}
\newtheorem{remark}[theorem]{Remark}

\numberwithin{equation}{section}

\begin{document}

\title{On the generic local Langlands correspondence for GSpin groups}

\author{Volker Heiermann}
\address{Aix Marseille Universit\'e, CNRS, Centrale Marseille, I2M, UMR 7373, 13453 Marseille, France}

\email{volker.heiermann@univ-amu.fr}

\author{Yeansu Kim}
\address{Department of Mathematics, University of Iowa, 14 MacLean Hall, Iowa city 52242, USA, Max Planck Institute for Mathematics, Vivatsgasse 7, 53111 Bonn, Germany}
\email{yeansu-kim@uiowa.edu, yskim@mpim-bonn.mpg.de}


\date{Dec 23, 2013}


\thanks{The first author has benefitted from help of the Agence
Nationale de la Recherche with reference ANR-08-BLAN-0259-02.}


\begin{abstract}
In the case of split $GSpin$ groups, we prove an equality of $L$-functions between automorphic local $L$-functions defined by the Langlands-Shahidi method and local Artin $L$-functions. Our method of proof is based on previous results of the first author which allow to reduce the problem to supercuspidal representations of Levi subgroups of $GSpin$, by constructing Langlands parameters for general generic irreducible admissible representations of $GSpin $ from the one for generic irreducible supercuspidal representations of its Levi subgroups.
\end{abstract}

\maketitle

\section{Introduction}
\label{Intro}

The main purpose of this paper is to show that the automorphic local $L$-functions for split $GSpin$ groups defined by the Langlands-Shahidi method are equal to Artin $L$-functions.  Our proof goes through the local Langlands correspondence using a method established by the first author (\cite{He1, He2}). Briefly, the local Langlands correspondence asserts that there exists a `natural' bijection between two different sets: Arithmetic objects (Galois or Weil-Deligne group) and analytic (automorphic) objects. The local Langlands correspondence is defined through the equality of Artin $L$-functions on the Galois side with automorphic $L$-functions.  For example, in the case of $GL_n$ (\cite{HT, H1}), the local Langlands correspondence is formulated by the equality of Artin $L$-functions and Rankin-Selberg $L$-functions for $GL_n \times GL_m$ (\cite{JPS, S8}). On the automorphic  side, Shahidi defined $L$-functions in the generic case using Eisenstein series. Rankin-Selberg $L$-functions are one of them (\cite{S8}).

In our paper, we study the Rankin product $L$-functions for $GL_m \times GSpin_{k}$ ($L$-functions from the Langlands-Shahidi method for $GSpin$ groups) using functoriality from $GSpin$ groups to general linear groups (\cite{AS1, AS2} and Section \ref{LP}). To explain our results more precisely, let $F$ denote a non-archimedean local field of characteristic zero. Let $\textbf{G}_n$ (resp. $\textbf{GL}_m$) denote the split general spin group of semisimple rank n, i.e. $\textbf{GSpin}_{2n+1}$ or $\textbf{GSpin}_{2n}$, (resp. general linear group of semisimple rank $m$) over $F$ and denote by $G_n$ (resp. $GL_m$) the group of $F$-points of $\textbf{G}_n$ (resp. $\textbf{GL}_m$). Let $\sigma \otimes \pi$ be an irreducible admissible generic representation of $M=\textbf{M}(F)$, where $\textbf{M}= \textbf{GL}_m \times \textbf{G}_n$ is the Levi subgroup of a standard parabolic subgroup $\textbf{P}$ in $\textbf{G}_{m+n}$. There is a given list of $L$-functions attached to $\sigma \otimes \pi$ defined by Shahidi. In the $GSpin$ case, we have two $L$-functions. The first $L$-function, denoted $L(s, \sigma \times \pi)$ (\cite{AS1}), is the Rankin product $L$-function for $GL_m \times G_n$. The second $L$-function  is either the twisted symmetric square $L$-function or the twisted exterior square $L$-function (Section \ref{Shahidi} for more details).

Note that Henniart \cite{H2} has recently proved that the twisted symmetric square and the twisted exterior square $L$-functions (the second $L$-functions) are Artin $L$-functions. In this paper, we prove the analogue result for the Rankin product $L$-functions for $GL_m \times G_n$ (the first $L$-functions). More precisely, we prove the following result (Theorem \ref{main}):

$ $

\noindent $\textbf{Theorem A.}$ The $L$-functions from the Langlands-Shahidi method in the case of $GSpin$ groups are Artin $L$-functions, i.e., $L(s, \sigma \times \pi)$ is an Artin $L$-function.

$ $

Our main theorem follows from results of the first author and the special case when $\pi$ is a supercuspidal representation. More precisely, in  \cite{He1} the first author constructed the Langlands parameters, the objects on the arithmetic side of the local Langlands correspondence, which correspond to admissible representations of a connected reductive group over $F$ after assuming the existence of the Langlands parameters that correspond to supercuspidal representations of its Levi subgroups. The equality between $L$-functions from the Langlands-Shahidi method and Artin $L$-functions holds, if this is true in the supercuspidal case \cite{He2}.

Note that in \cite{He1}, first the Langlands parameters of discrete series representations of connected reductive groups over $F$ are constructed from the Langlands parameters of supercuspidal representations of Levi subgroups. The proof is based on the classification of the supercuspidal support of discrete series representations in terms of poles of Harish-Chandra's $\mu $-function given in \cite{He3}. All this is subject to some additional assumptions on the Langlands parameters of supercuspidal representations. However, these assumptions are satisfied for a generic discrete series representation if the $L$-functions of a representation in its supercuspidal support agree with the corresponding Artin $L$-functions of its Langlands parameter. One has also that the $\gamma $-factors of a generic discrete series representation agree with the $\gamma $-factors of its Langlands parameter, if the  equality of $\gamma $-factors is true for a representation in its supercuspidal support.

The Langlands parameters of arbitrary admissible irreducible representations are deduced from this with help of the Langlands classification and equality of $L$-functions and $\gamma $-factors holds under the analog conditions.

Therefore, by the results in \cite{He1, He2}, it is possible to reduce Theorem A to the existence of Langlands parameters for generic supercuspidal representations with equality of $L$-functions (i.e. Theorem A in the case of supercuspidal representations). More precisely, let $\pi_{sc}$ be an irreducible generic supercuspidal representation of ${G}_n$ and let $\Pi_{sc}$ be an irreducible admissible representation of $GL_{2n}$ which is the local functorial lift of $\pi_{sc}$ constructed in \cite{AS1}. We prove (Theorem \ref{F:local2})

$ $

\noindent $\textbf{Theorem A'}$ For any irreducible generic supercuspidal representation $\sigma_{sc}$ of $GL_m$ we have
$$L(s, \sigma_{sc} \times \pi_{sc}) = L(s, \sigma_{sc} \times \Pi_{sc}).$$

$ $

The $L$-functions on the right hand side, i.e. $L(s, \sigma_{sc} \times \Pi_{sc})$, are the Rankin-Selberg $L$-functions. The Rankin-Selberg $L$-functions are Artin $L$-functions due to the local Langlands correspondence for $GL_n$ \cite{HT, H1}. Therefore, Theorem A' implies  that the Rankin product $L$-functions for $GL_m \times G_n$ in the generic supercuspidal case are Artin $L$-functions.

Let us briefly explain the proof of Theorem A' here. Due to \cite[Proposition 5.1]{S2}, an irreducible generic supercuspidal representation $\pi_{sc}$ can be embedded into a globally generic cuspidal representation. Therefore, we can use the global functional equation \cite{S2} to get the equality of $\gamma$-factors. Since $L$-functions from the Langlands-Shahidi method are completely determined by corresponding $\gamma$-factors in the tempered case (Definition \ref{Shahidi's L-function}), it is enough to show that the functorial lift $\Pi_{sc}$ of $\pi_{sc}$ is a tempered representation. (Theorem \ref{F:local2})

Furthermore, we describe the local transfer image of the functorial lift of a generic supercuspidal representation $\pi_{sc}$ of $G_n$ (Theorem \ref{F:main}) to show that the Langlands parameter of the local functorial lift $\Pi_{sc}$ of $\pi_{sc}$ factors through the $L$-group of $\textbf{G}_n$ (Theorem \ref{F:Langlands parameter}). This Langlands parameter can then be considered as the Langlands parameter of the generic supercuspidal representation $\pi_{sc}$ of $G_n$.

$ $

\noindent $\textbf{Remark}$ We also prove the equality of $\gamma$-factors in Theorem A and Theorem A'. More precisely, let $\psi_F$ be
a fixed non trivial additive character of $F$ and $\gamma(s, \sigma \times \pi , \psi_F)$ the complex function defined in \cite[Theorem 3.5]{S2}, where $\pi$ and $\sigma$ are as in Theorem A. Then, the $\gamma(s, \sigma \times \pi, \psi_F)$'s are Artin $\gamma$-factors. Furthermore, let $\pi_{sc}$, $\sigma_{sc}$ and $\Pi_{sc}$ be as in Theorem A'. Then, $\gamma(s, \sigma_{sc} \times \pi_{sc}, \psi_F) = \gamma(s, \sigma_{sc} \times \Pi_{sc}, \psi_F)$.

$ $

Our main result (Theorem A) has an interesting application to the structure of $L$-packets, the partition of the set of (equivalence classes of) irreducible admissible representations  of a quasi-split reductive group $\textbf{H}$ over a non-archimedean local field $F$ given by the (in general conjectural) Langlands correspondence. More precisely, denote by $W_F'$ the Weil-Deligne group, i.e. $W_F'=W_F \times SL(2,\mathbb{C})$,where $W_F$ is the Weil group. Let $\psi$ be an Arthur parameter for $\textbf{H}$, i.e. a map $W_F' \times SL(2, \mathbb{C}) \rightarrow \ ^L\textbf{H}$ that satisfies certain properties (See \cite{S4} for more details). There is a Langlands parameter $\phi_{\psi}$ which corresponds to the Arthur parameter $\psi$ which is defined by $\phi_{\psi}(w)=\psi(w,\begin{pmatrix} |w|^{\frac{1}{2}} & 0 \\ 0 & |w|^{-\frac{1}{2}} \end{pmatrix})$. We consider the $L$-packet which is attached to this Langlands parameter $\phi_{\psi}$. The following conjecture is called the generic Arthur packet conjecture (\cite{S4}):

${ }$

\noindent $\textbf{Conjecture.}$ Let $\Pi(\phi_{\psi})$ be the $L$-packet attached to the Langlands parameter $\phi_{\psi}$ which corresponds to the Arthur parameter $\psi$ of $\textbf{H}$. Suppose that $\Pi(\phi_{\psi})$ has a generic member. Then it is a tempered $L$-packet.

${ }$

\noindent $\textbf{Remark}$ This conjecture is first formulated in \cite{S4} for any connected reductive group and strengthened in \cite{Kim} for classical groups and $GSpin$ groups. This conjecture can be considered as a local version of the generalized Ramanujan conjecture and the conjecture itself is related to the generalized Ramanujan conjecture (\cite{Ar}).

${ }$

Shahidi \cite[Theorem 5.1]{S4} proved that, if the equality of $L$-functions (Theorem A in the $GSpin$ case) through the local Langlands correspondence holds, then the generic Arthur packet conjecture is also true. Therefore, our main theorem implies the following:

${ }$

\noindent $\textbf{Theorem B.}$ The generic Arthur packet conjecture is true for split $GSpin$ groups.
${ }$

\noindent $\textbf{Remark}$ The generic Arthur packet conjecture is also true in general when $F$ is archimedean local field since the full local Langlands conjecture for real groups is a theorem in \cite{L}. The equality of $L$-functions in this case is proved by Shahidi in \cite{S5}.\\

The paper is organized as follows. In Section \ref{N}, we recall the standard notations.

In Section \ref{Pre}, we briefly explain $L$-functions from the Langlands-Shahidi method in the case of $GSpin$ groups (see \cite{S0, S6, S5, S1, S2} for Langlands-Shahidi method for any connected reductive groups). We first introduce $\gamma$-factors and explain how $L$-functions are defined in the tempered case. Then, $L$-functions are defined using Langlands classification in the arbitrary case (Definition \ref{Shahidi's L-function}). In this section, we also explain the multiplicativity of $\gamma$-factors and the multiplicativity of $L$-functions in the general setting. The multiplicativity property is one of the essential tools in our paper. We give several examples of the multiplicativities of local factors in the case of $GSpin$ groups which support the proofs of theorems in Section \ref{LP} (Examples \ref{Multi:GSpin} and \ref{Multi:GL}). We also explain the global functorial lift from $GSpin$ groups to general linear groups obtained by Asgari and Shahidi \cite{AS1, AS2}.

In Section \ref{LP}, we construct a local Langlands parameter that corresponds to any irreducible admissible generic representations of $GSpin$ groups with the equality of $L$-functions through the local Langlands correspondence. More precisely, in subsection \ref{sc:L}, we first show the equality of $L$-functions in the case of supercuspidal generic representations of $GSpin$ groups (Theorem \ref{F:local2}). In subsection \ref{sc:LP}, we use the equality of $L$-functions to describe the image of the local functorial lift of supercuspidal generic representations of $GSpin$ groups (Theorem \ref{F:main}). Then, we construct a local Langlands parameter that corresponds to any irreducible admissible representations of $GSpin$ groups (Theorem \ref{F:Langlands parameter}). Note that $L$-function condition in Theorem \ref{F:main} is essential in the proof of Theorem \ref{F:Langlands parameter}. In subsection \ref{H}, we generalize the result on the construction of local Langlands parameters in the supercuspidal case to the case of irreducible admissible generic representations. Here, we apply Heiermann's construction of local Langlands parameters \cite{He3, He1, He2} to our case (Theorem \ref{F:reduction}). We then show the equality of $L$-functions through the local Langlands correspondence (Theorem \ref{main}), which is the main result of this paper and implies the generic Arthur packet conjecture for split $GSpin$ groups as explained above.

In Section \ref{Erratum}, we add an erratum to \cite{He1}.

\section{Notation}
\label{N}
Let $F$ denote a non-archimedean local field of characteristic zero, either archimedean or non-archimedean, $k$ a number field (global field of characteristic zero) and $\mathbb{A}=\mathbb{A}_k$ its ring of adeles. If needed, we let $k_{v_0}=F$ for a place $v_0$ of $k$. Let $\textbf{G}_n$ (resp. $\textbf{GL}_m$) denote the split general spin group of semisimple rank n, i.e. $\textbf{GSpin}_{2n+1}$ or $\textbf{GSpin}_{2n}$, (resp. general linear group of semisimple rank $m$) over $F$ or over $k$ and write $^L\textbf{G}_n$ for the $L$-group of $\textbf{G}_n$. In the case of split $GSpin$ groups, $^L\textbf{G}_n$ is either $\textbf{GSp}_{2n}(\mathbb{C})$ or $\textbf{GSO}_{2n}(\mathbb{C})$.
Let $\textbf{M}$ be the Levi subgroup of a standard parabolic subgroup $\textbf{P}=\textbf{M}\textbf{N}$ in $\textbf{G}_n$. It corresponds to an ordered partition $s=(n_1, n_2, \ldots, n_k)$ of some $n'$ with $n' \leq n$ and $n-n' \neq 1$.  (The Levi subgroup $\textbf{M}_s$ associated to the partition $s$ is isomorphic to
$$\textbf{GL}_{n_1} \times \textbf{GL}_{n_2} \times \cdots \times \textbf{GL}_{n_k} \times \textbf{G}_{n-n'} \ (\text{see \cite{A}} ).$$

For any connected reductive group $\textbf{G}$ defined over $F$ or $k$, we denote by $G$ (resp. $\underline{G}$) its group of  $F$-rational points (resp. $\mathbb{A}$-rational points). To differentiate global representations from local representations, we use $\underline{ \ \ }$ to denote global representations. For example, $\underline{\pi}$ in Sec \ref{F:global} will denote an automorphic representation of $\underline{G}_n$.

We denote the normalized induced representation $i_{P_s}^{G_n}(\rho_1 \otimes \cdots \otimes \rho_k \otimes \pi)$ by
$$\rho_1 \times \cdots \times \rho_k \rtimes \pi$$
\noindent where $P_s=M_sN$ is the standard $F$-parabolic subgroup of $G_n$ which corresponds to the partition $s$ and each $\rho_i$ (resp. $\pi$) is a representation of some $GL_{n_i}$ (resp. $G_{n-n'}$). In particular, $i_P^G$ is a functor from admissible representations of $M$ to admissible representations of $G$ that sends unitary representations to unitary representations. We also denote the normalized induced representation $i_{P}^{GL_n}(\rho_1 \otimes \cdots \otimes \rho_k)$ by
$$\rho_1 \times \cdots \times \rho_k$$
\noindent where $P=MN$ is the standard $F$-parabolic subgroup of $GL_n$ where $\textbf{M} \cong \textbf{GL}_{n_1} \times \textbf{GL}_{n_2} \times \cdots \times \textbf{GL}_{n_k}$ and each $\rho_i$ is a representation of some $GL_{n_i}$.

For irreducible admissible generic representations $\sigma$ and $\sigma'$ of $GL_n$ and $GL_{n'}$, respectively and for an irreducible admissible generic representation $\pi$ of $G_m$, the symbol 
\begin{equation}\label{Rankin product L-function}
L(s, \sigma \times \pi) \ \  ({\rm resp.} \ L(s, \sigma \times \sigma'))
\end{equation}
will denote the local Rankin product $L$-function for $GL_m \times G_n$ (resp. the local Rankin-Selberg $L$-function) and the corresponding $\gamma $-factor will be denoted 
\begin{equation}\label{Rankin product gamma-factor}
\gamma(s, \sigma \times \pi, \psi_{F}) \ \ ({\rm resp.} \ \gamma(s, \sigma \times \sigma', \psi_{F})).
\end{equation}

\section{Preliminaries}
\label{Pre}

\subsection{$L$-functions from the Langlands-Shahidi method}
\label{Shahidi}
$ $

Let us briefly explain how Shahidi defines local $L$-functions in the case of $GSpin$ groups.
Let $\textbf{M} \cong \textbf{GL}_m \times \textbf{G}_n$ be the Levi subgroup of a maximal standard parabolic subgroup $\textbf{P}=\textbf{M}\textbf{N}$ of $\textbf{G}_{m+n}$ and let $\sigma \otimes \pi$ be an irreducible admissible generic representation of $M=\textbf{M}(F)=GL_m\times G_n$. Denote by $\rho_m$ the standard representation of $GL_m(\mathbb{C})$, by $\widetilde{R}$ the contragredient of the standard representations of $\,^{L}\textbf{G}_n$ and by $\mu$ the similitude character of $\,^{L}\textbf{G}_n$. The adjoint action $r$ of $\,^{L}\textbf{M}$, the $L$-group of $M$, on $\,^{L}\mathfrak{n}$, the Lie algebra of the $L$-group of $N$ decomposes as $r=r_2 \ \text{or} \ r_1 \oplus r_2$ (\cite[Proposition 5.6]{A} or \cite[chapter 3]{AS1}),where
\noindent \[ \left\{ \begin{array}{cc}
 r_1= \rho_m \otimes \widetilde{R} \ \ \ \text{and} \ \ \ r_2=Sym^2 \otimes \mu^{-1} &  \text{if} \ \ \ \textbf{G}_n = \textbf{GSpin}_{2n+1} \\
  r_1= \rho_m \otimes \widetilde{R} \ \ \ \text{and} \ \ \ r_2=\wedge^2 \otimes \mu^{-1} &  \text{if} \ \ \ \textbf{G}_n = \textbf{GSpin}_{2n}
\end{array} \right. \]
Shahidi first defined two complex functions $\gamma(s, \sigma \otimes \pi, r_i, \psi_{F})$, called $\gamma $-functions, attached to $\sigma \otimes \pi$ and $r_i$ for $i=1,2$, where $\psi_F$ is a fixed non trivial additive character of $F$ (\cite[Theorem 3.5]{S2}). Using $\gamma$-factors, Shahidi defined corresponding $L$-functions attached to $\sigma \otimes \pi$ and $r_i$, denoted $L(s, \sigma \otimes \pi, r_i)$ for $i=1,2$, in \cite{S2} by the following way:

\begin{definition}\label{Shahidi's L-function}
If $\sigma$ is tempered, define $L$-functions from the Langlands-Shahidi method as inverse of the normalized numerator of the $\gamma$-factors. One defines then $L$-functions for an arbitrary irreducible admissible generic representation using Langlands classification (\cite{Si1}). The first $L$-function, i.e. $L(s, \sigma \otimes \pi, r_1)$, is the Rankin product $L$-function for $GL_m \times G_n$ and will consequently be denoted $L(s, \sigma \times \widetilde{\pi})$ in the sequel. We will also write $\gamma(s, \sigma \times \pi, \psi_F)$ for $\gamma(s, \sigma \otimes \tilde{\pi}, r_1, \psi_{F})$. The second $L$-function, i.e. $L(s, \sigma \otimes \pi, r_2)$, is either the twisted symmetric square $L$-function or the twisted exterior square $L$-function. 
\end{definition}

\begin{remark}
In the tempered case, the equality of $\gamma$-factors implies the equality of $L$-functions since the $L$-functions are completely determined by $\gamma$-factors in the tempered case.
\end{remark}

There are several equivalent definitions of the notion of a tempered representation. Let us introduce one of them that we use throughout the paper.

\begin{definition}
Let $\textbf{G}$ be a connected reductive group over $F$ and put $G=\textbf{G}(F)$. An irreducible admissible representation $\pi$ of $G$ is called tempered if there exists an irreducible square-integrable, i.e. a discrete series representation, $\sigma$ of a Levi subgroup $M$ of $G$ such that $\pi$ can be embedded into $i_P^G\sigma $.
\end{definition}

We introduce one important property below which is the multiplicativity of $\gamma$-factors. Let us first explain the general setting. We fix a Borel subgroup $\bf B = \bf TU$ of a quasi-split connected reductive group $\bf G$ over $F$, where $\bf T$ is a maximal torus and $\bf U$ denotes the unipotent radical of $\bf B$. Let ${\bf A}_0 \subset \bf T$ be the maximal split subtorus of $\bf T$ (which is one of $\bf G$) and let $\Phi$ be the set of $F$-roots of ${\bf A}_0$. Then, we can write $\Phi = \Phi^+ \cup \Phi^-$, where $\Phi^+$ is the set of positive roots with respect to $\bf U$. We also let $\Delta \subset \Phi^+$ be the set of simple roots. Let $\bf P = \bf MN$ be a standard parabolic subgroup of $\bf G$, i.e. $\bf P$ is a parabolic subgroup that contains $\bf B$, $\bf N$ is its unipotent radical and $\bf M$ the unique Levi subgroup that contains $\bf T$. If $\theta$ is the subset of $\Delta$ that generates $\bf M$, then we denote such $\bf P$ (resp. $\bf M$) by ${\bf P}_{\theta}$ (resp. ${\bf M}_{\theta}$). It is well known that this establishes a bijection between subsets of $\Delta $ and standard parabolic subgroups $\bf P$ of $\bf G$.

Now, we are ready to explain the multiplicativity of $\gamma$-factors in the general setting. Assume $\sigma $ is an irreducible admissible generic representation of $M$, $\sigma \subset i_{P'}^M \sigma'$, where $\textbf{P}'=\textbf{M}' \textbf{N}'$ is a standard parabolic subgroup of $\textbf{M}$ with respect to $\textbf{M}\cap\textbf{B}$ and $\sigma'$ is an irreducible admissible generic representation of $M'$. Let $\theta_1 \subset \Delta$ be such that $\bf M' = \bf M_{\theta_1}$. Let $\widetilde{\omega}:=\widetilde{\omega}_{l,\Delta}\widetilde{\omega}_{l,\theta_1}$, where $\widetilde{\omega}_{l,\Delta}$ (resp. $\widetilde{\omega}_{l,\theta_1}$) be the longest element in the Weyl group of $\textbf{A}_0$ in $\textbf{G}$ (resp. $\textbf{M}$). Lemma 4.2.1 \cite{S7} implies that there exists a set of subsets $\theta_2, \theta_3, \cdots, \theta_n \subset \Delta$ such that

\begin{itemize}
\item $\theta_n = \widetilde{\omega}(\theta_1)$;
\item fix $1 \leq j \leq n-1$; then there exists $\alpha_j \in \Delta - \theta_j$ such that $\theta_{j+1}=\widetilde{\omega}_{l,\Omega_j}\widetilde{\omega}_{l,\theta_j}(\theta_j)$, where $\Omega_j = \theta \cup \{ \alpha_j \}$ and $\widetilde{\omega}_{l,\Omega_j}$ (resp. $\widetilde{\omega}_{l,\theta_j}(\theta_j)$) are the longest elements in the Weyl group of ${\bf A}_0$ in ${\bf M}_{\Omega_j}$ (resp. ${\bf M}_{\theta_j}$);
\item set $\widetilde{\omega}_{j}=\widetilde{\omega}_{l,\Omega_j}\widetilde{\omega}_{l,\theta_j}$ for $j=1, \cdots, n-1$, then $\widetilde{\omega}= \widetilde{\omega}_1 \cdots \widetilde{\omega}_{n-1}$.
\end{itemize}

For each $j$, $2 \leq j \leq n-1$, let $\overline{\omega}_j=\widetilde{\omega}_{j-1} \cdots \widetilde{\omega}_{1}$. The group $\textbf{M}_{\Omega_i}$ contains $\textbf{M}_{\theta_j}$ as a Levi factor of a maximal parabolic subgroup and $\overline{\omega}_j(\sigma')$ is a representation of $M_{\theta_j}$.

Fix now an irreducible component $r_i$ of the action of $\,^{L}\textbf{M}$ on the Lie algebra of $\,^{L}\textbf{N}$ and denote $V_i$ its space. The group $\,^{L}\textbf{M}'$ acts on $V_i$ as defined by $\,^{L}\textbf{M}$. Given an irreducible constituent of the action, there exists a unique $j, 1 \leq j \leq n-1$, which is sent under $\overline{\omega}_j$ to an irreducible constituent of the action of $\,^{L}\textbf{M}_{\theta_j}$ on the Lie algebra of $\,^{L}\textbf{N}_{\theta_j}$ whose adjoint action is denoted by $r_{i(j)}'$. We denote by $S_i$ the set of all such $j$, which is in general a proper subset of $1 \leq j \leq n-1$. 

Shahidi has associated in \cite{S2}, Theorem 3.5, to $\sigma $ and $r_i$ as above $\gamma $-factors in general. The $\gamma $-factors are uniquely determined by several properties. One important property, called multiplicativity, given in part 3 of Theorem 3.5 of \cite{S2} reads in the above setting:

\begin{theorem}[Multiplicativity of $\gamma$-factors]
\label{Multiplicativity}

For each $j \in S_i$, let $\gamma(s, \overline{\omega}_j(\sigma'), r_{i(j)}', \psi_{F})$ be the corresponding $\gamma$-factor.
Then,
$$\gamma(s, \sigma , r_i, \psi_{F})=\displaystyle\mathop{\prod}\limits_{j \in S_i} \gamma(s, \overline{\omega}_j(\sigma'), r_{i(j)}', \psi_{F}).$$

\end{theorem}

The following example gives the multiplicativity of $\gamma$-factors in the case of $GSpin$ groups. This example is used in the proof of Theorem \ref{F:reduction}. Let us first recall that $\gamma(s, \rho \times \pi, \psi_F)$ is the notation for the $\gamma$-factor  $\gamma(s, \rho \otimes \tilde{\pi}, r_1, \psi_F)$ (Definition \ref{Shahidi's L-function} and (\ref{Rankin product gamma-factor})): 

\begin{example}\label{Multi:GSpin}
Let $\pi$ be an irreducible admissible representation of $G_n$. By Jacquet's quotient theorem, there exists a standard parabolic subgroup $\textbf{Q}=\textbf{L}\textbf{V}$ of $\textbf{G}_n$, $\textbf{L}=\textbf{GL}_{n_1} \times \cdots \times \textbf{GL}_{n_k} \times \textbf{G}_{n-n'}$, and an irreducible supercuspidal representation $\rho_1 \otimes \rho_2 \otimes \cdots \otimes \rho_r \otimes \pi_{cusp}$ of $L$ such that $\pi$ is a subrepresentation of $\rho_1 \times \rho_2 \times \cdots \times \rho_r \rtimes \pi_{cusp}$. (Here, each $\rho_i$ is an irreducible supercuspidal representation of some $GL_{n_i}$ and $\pi_{cusp}$ is an irreducible supercuspidal representation of $G_{n-n'}$.) Then, for any irreducible supercuspidal representation $\rho$ of $GL_m$, we have
$$\gamma(s, \rho \times \pi, \psi_F) = \gamma(s, \rho \times \pi_{cusp}, \psi_F) \ \displaystyle\prod\limits_{i=1}^{r} \gamma(s , \rho \times \rho_i, \psi_F) \gamma(s, \rho \times (\widetilde{\rho_i}\otimes \omega_{\pi}), \psi_F).$$
where $\omega_{\pi}$ is the central character of $\pi$.
\end{example}

Remark that the multiplicativity of $L$-functions does not hold in general since the $L$-functions in the numerator and those in the denominator of $\gamma$-factors might cancel each other. However, there are several cases of which multiplicativity of $L$-functions holds. Here, we include the following examples for the multiplicativity of $L$-functions in the case of general linear groups. These examples are used in the proof of Theorem \ref{F:main}.

\begin{example}\label{Multi:GL}

(a) (Proposition on page 351 of \cite{JPS}) Let $\Pi$ (resp. $\Pi'$) be an irreducible tempered representation of $GL_n$ (resp. $GL_{n'}$) of the form $\delta_1 \times \cdots \times \delta_d$ (resp. $\delta_1' \times \cdots \times \delta_{d'}'$), where each $\delta_i$ (resp. $\delta_j'$) is a discrete series representation of some $GL_{n_i}$ (resp. $GL_{n_j'}$). Then, we have
$$L(s, \Pi \times \Pi') = \displaystyle\prod\limits_{i, j} L(s, \delta_i \times \delta_j').$$

\noindent (b) (Theorem on page 444 of \cite{JPS}) Let $\delta$ be a discrete series representation of $GL_n$ which can be realized as the unique irreducible subrepresentation, denoted $\delta([\nu^{-\frac{t-1}{2}}\rho, \nu^{\frac{t-1}{2}}$ $ \rho])$, of $\nu^{\frac{t-1}{2}} \rho \times \cdots \times \nu^{-\frac{t-1}{2}} \rho$ with $\rho$ a unitary supercuspidal representation of some $GL$ (\cite{ZB, Z}). Then, we have
$$L(s, \delta \times \widetilde{\delta}) = \displaystyle\prod\limits_{i=0}^{t-1} L(s + i, \rho \times \widetilde{\rho}).$$
\end{example}

Let us conclude this Section by recalling one of the important global properties of the theory of $L$-functions which is the functional equations.

\begin{theorem}[Global Functional Equation]
\label{GFE}
Let $\underline{\pi}:=\otimes \pi_v$ be a (global) generic cuspidal representation of $\underline{G}_n =\textbf{G}_n(\mathbb{A})$. Let $\gamma(s, \sigma_v \otimes \pi_v, r_{i,v}, \psi_{k_v})$ and $L(s, \sigma_v \otimes \pi_v, r_{i,v})$ (for $i=1, 2$) be the local factors defined by Shahidi in the case of GSpin groups (Definition \ref{Shahidi's L-function}). Let $S$ be a finite set of places of $k$ such that for $v \notin S$, $\textbf{G}_n \times_k k_v$, $\pi_v$ and $\chi_v$ are all unramified.
Then we have the following equality:
$$\displaystyle\prod\limits_{v \notin S} L(s, \pi_v, r_{i,v}) =
\displaystyle\prod\limits_{v \in S} \gamma(s, \sigma_v \otimes \pi_v, r_{i,v}, \psi_{k_v}) \displaystyle\prod\limits_{v \notin S} L(1-s, \pi_v, \widetilde{r_{i,v}}) $$
\end{theorem}
$ $

\subsection{Global functoriality}
$ $

Asgari and Shahidi studied the functorial lifts from automorphic representations of $\underline{G}_n$ to $\underline{GL}_{2n}$ (\cite{AS1} and \cite{AS2}). To state the main theorem in \cite{AS1, AS2}, let $k$ be global field and fix a Borel subgroup $\textbf{B}$ in $\textbf{G}_n$ with a maximal split torus $\textbf{T}$, and denote the associated roots by $R$ and the positive roots by $R^+$. For each $\alpha \in R$ denote the root group homomorphism associated with $\alpha$ by $u_{\alpha}: \mathbb{G}_a \rightarrow \textbf{G}$. We also denote unipotent radical of $\textbf{B}$ by $\textbf{U}$. To define a generic character of $\textbf{U}(k) \backslash \underline{U}$, we fix a splitting, i.e., the choice of Borel subgroup along with a collection of root vectors in $\textbf{U}$, one for each simple root of $\textbf{T}$. Then, using a fixed splitting, every $u \in \underline{U}$ can be written uniquely as $u = \prod u_{\alpha}(x_{\alpha})$, where the product runs over all simple roots. Let $\psi_k$ be a non-trivial continuous character of $k \backslash \mathbb{A}$. Then $\psi_k = \otimes \psi_{k_v}$, where each $\psi_{k_v}$ is a non-trivial additive character of $k_v$. Then we define a generic character $\chi$ of $\textbf{U}(k) \backslash \underline{U}$ as
$$\chi(u)=\psi_k( \displaystyle\sum\limits_{\alpha}x_{\alpha})$$
where the sum runs over all simple roots.

Now, we are ready to define the globally generic representations. An irreducible automorphic cuspidal representation $(\underline{\pi}, V_{\underline{\pi}})$ of $\underline{G}_n$ is called globally generic if there exists a cusp form $f \in V_{\underline{\pi}}$ such that
$$ \int_{\textbf{U}(k) \backslash \underline{U}} f(ng) \chi^{-1}(n) dn \neq 0.$$

The following theorem is Theorem 3.5 of \cite{AS2}:

\begin{theorem}[\cite{AS1}, \cite{AS2}]\label{F:global}

Let $\underline{\pi} = \otimes \pi_v$ be a globally generic, irreducible, cuspidal, automorphic representation of $\underline{G}_n$ with central character $\omega_{\underline{\pi}}$. Write $\chi=\otimes \chi_v$. Let $S$ be a nonempty finite set of non-archimedean places $v$ such that, for $v \notin S$, we have that $\pi_v$ and $\chi_v$ are unramified. Then $\underline{\pi}$ has a unique functorial transfer to an automorphic representation $\underline{\Pi}=\otimes \Pi_v$ of $\underline{GL}_{2n}$ such that, for all $v \notin S$, the homomorphism parametrizing the local representation $\Pi_v$ is given by
\begin{equation}\label{homomorphism}
\Phi_v = \iota \circ \phi_v : W_v \rightarrow GL_{2n}(\mathbb{C}),
\end{equation}
where $W_v$ denotes the local Weil group of $k_v$, $\phi_v : W_v \rightarrow \,^{L}\textbf{G}_n$ the homomorphism parametrizing $\pi_v$ and $\iota: \,^{L}\textbf{G}_n\hookrightarrow GL_{2n}(\mathbb{C})$ the natural embedding. Moreover, the transfer $\underline{\Pi}$ satisfies
$$\underline{\Pi} \cong \widetilde{\underline{\Pi}} \otimes \omega_{\underline{\pi}} \ \ \text{and} \ \ \omega_{\underline{\Pi}} = \omega_{\underline{\pi}}^n,$$
The automorphic representation $\underline{\Pi}$ is an isobaric sum of the form
\begin{equation}\label{isobaric}
\underline{\Pi}=\underline{\Pi}_1 \times \cdots \times \underline{\Pi}_t= \underline{\Pi}_1 \boxplus \cdots \boxplus \underline{\Pi}_t,
\end{equation}
where each $\underline{\Pi}_i$ is a unitary, cuspidal representation of $\underline{GL}_{n_i}$ such that, for T a sufficiently large finite set of places of $k$ containing the archimedean places, the partial $L$-function $L^T(s, \underline{\Pi}_i, \wedge^2 \otimes \omega_{\underline{\pi}})$ has a pole at $s=1$ in the odd case and $L^T(s, \underline{\Pi}_i, Sym^2 \otimes \omega_{\underline{\pi}})$ has a pole at $s=1$ in the even case. We have $\underline{\Pi}_i \ncong \underline{\Pi}_j$ if $i \neq j$, and $n_1 + \cdots + n_t = 2n$ with each $n_i > 1$.

\end{theorem}

\section{The equality of $L$-functions and Langlands parameter}
\label{LP}

\subsection{The equality of $L$-functions: supercuspidal case}
\label{sc:L}
$ $

In this Section, we adapt the idea in \cite{CKSS1, CKSS2} to generalize it to $GSpin$ groups. In the proof of Theorem \ref{F:global}, using the natural embedding  (\ref{homomorphism}), Asgari and Shahidi constructed archimedean and non-archimedean unramified transfer with equality of $\gamma$-factors and $L$-functions in those cases (Chapter 6 of \cite{AS1} or Proposition 3.20 and Proposition 3.25 of \cite{AS2}):


\begin{proposition}[\cite{AS1, AS2}]\label{F:S}
Let $v_0$ be an archimedean or non-archimedean place of $k$. Let $\pi$ be an irreducible, admissible, generic representation of $G_n$. We further assume that $\pi$ is unramified if $v_0$ is a non-archimedean place. Then, there exists a unique local functorial lift $\Pi$ of $\pi $ to $GL_{2n}$ such that, for any irreducible, admissible, generic representation $\sigma$ of $GL_{m}$, we have
$$\gamma(s, \sigma \times \pi, \psi_{k_{v_0}}) = \gamma(s, \sigma \times \Pi, \psi_{k_{v_0}}) \ \ \text{and} \ \ L(s, \sigma \times \pi) = L(s, \sigma \times \Pi).$$
\end{proposition}

Now, we are ready to show the equality of $L$-functions for any place $v$ of $k$.

\begin{theorem}\label{F:local2}

Let $v_0$ be a non-archimedean place of $k$ and let again $k_{v_0}=F$. Let $\pi$ be an irreducible unitary generic supercuspidal representation of $G_n$. Then there exists a unique generic representation $\Pi$ of $GL_{2n}$ such that for every irreducible unitary supercuspidal representation $\rho$ of $GL_m$ we have
$$\gamma(s, \rho \times \pi, \psi_F) = \gamma(s, \rho \times \Pi, \psi_F) \ \ \text{and} \ \ L(s, \rho \times \pi) = L(s, \rho \times \Pi).$$\label{equality}
In particular, $\Pi = \delta_{1} \times \cdots \times \delta_{d}$ with each $\delta_{i}$ a discrete series representation, i.e., $\Pi$ is tempered.

\end{theorem}

\begin{proof}


We can embed $\pi$ (resp. $\rho$) into a (global) generic automorphic cuspidal representation $\underline{\pi}=\otimes \pi_v$ (resp. $\underline{\rho}= \otimes \rho_v$) such that $\pi_w$ (resp. $\rho_w$) is unramified for all other finite places $w \neq v_0$ and $\pi_{v_0}=\pi$ (resp. $\rho_{v_0} = \rho$) (\cite[Proposition 5.1]{S2}). In the archimedean or non-archimedean unramified case, i.e. for $w \neq v_0$, we know the equality of $\gamma$-factors (Proposition \ref{F:S}). The global functional equation (Theorem \ref{GFE}) implies that
$$\gamma(s, \rho \times \pi, \psi_F) = \gamma(s, \rho \times \Pi, \psi_F)$$
for every irreducible supercuspidal representation $\rho $ of $GL_m$.

Using the multiplicativity of $\gamma$-factors, we get the equality $\gamma(s, \sigma \times \pi, \psi_F) = \gamma(s, \sigma \times \Pi, \psi_F)$
for a discrete series representation $\sigma$. The discrete series representation $\sigma$ can be realized as the unique irreducible subrepresentation $\delta([\nu^{-\frac{t-1}{2}}\rho, \nu^{\frac{t-1}{2}}$ $ \rho])$ of $\nu^{\frac{t-1}{2}} \rho \times \cdots \times \nu^{-\frac{t-1}{2}} \rho$ with $\rho$ a unitary supercuspidal representation of some $GL$ (\cite{ZB, Z}).
Since $\Pi$ is generic and unitary, $\Pi$ can be considered as full induced representation $\nu^{r_1}\delta_{1} \times  \cdots \times \nu^{r_k}\delta_{k} \times \delta_{k+1} \times \cdots \times \delta_{k+l} \times \nu^{-r_k}\delta_{k} \times \cdots \times \nu^{-r_1}\delta_{1}$ by the classification of unitary generic representations of $GL$ with each $\delta_{i}$ a discrete series representation of some $GL$ and $0 < r_k \leq \cdots \leq r_1 < \frac{1}{2}$ (\cite{T4}). Using the multiplicativity of $\gamma$-factors and the definition of $\gamma$-factors,

$$\gamma(s, \sigma \times \pi, \psi_F) = \gamma(s, \sigma \times \Pi, \psi_F)$$

$=\displaystyle\prod\limits_{j=1}^{k} \gamma(s+r_j, \sigma \times \delta_{j}, \psi_F) \gamma(s-r_j, \sigma \times \delta_{j}, \psi_F) \times \displaystyle\prod\limits_{i=1}^{l} \gamma(s, \sigma \times \delta_{k+i}, \psi_F).$

$\sim \frac{(\displaystyle\prod\limits_{j=1}^{k} L(s+r_j, \sigma \times \delta_{j})L(s-r_j, \sigma \times \delta_{j})\displaystyle\prod\limits_{i=1}^{l} L(s, \sigma \times \delta_{k+i}))^{-1}}{(\displaystyle\prod\limits_{j=1}^{k} L(1-s-r_j, \widetilde{\sigma} \times \widetilde{\delta_{j}})L(1-s+r_j, \widetilde{\sigma} \times \widetilde{\delta_{j}})\displaystyle\prod\limits_{i=1}^{l} L(1-s, \widetilde{\sigma} \times \widetilde{\delta_{k+i}}))^{-1}}$

where $\sim$ means that the two factors are equal up to a monomial in $q^{-s}$.

We can define $L(s, \sigma \times \pi)$ as the inverse of the normalized numerator of $\gamma(s, \sigma \times \pi, \psi_F)$ since $\sigma \times \pi$ is tempered (Definition \ref{Shahidi's L-function}). Furthermore, we know that the numerator and the denominator are relatively prime (see \cite{JPS}). Therefore,
$$L(s, \sigma \times \pi) =  \displaystyle\prod\limits_{j=1}^{k} L(s+r_j, \sigma \times \delta_{j})L(s-r_j, \sigma \times \delta_{j})\displaystyle\prod\limits_{i=1}^{l} L(s, \sigma \times \delta_{k+i}).$$
We consider the equality of $L$-functions above with $\sigma = \widetilde{\delta_{i}}$ for $i=1, \ldots, k$. The tempered $L$-function conjecture \cite[Theorem 5.7]{A} implies that $L(s, \widetilde{\delta_{i}} \times  \pi)$ is holomorphic for $Re(s)>0$. On the other hand, the right hand side has a pole at $s=r_i>0$ since $L(s-r_i, \widetilde{\delta_{i}} \times \delta_{i})$ has a pole at $s=r_i>0$. This implies that $k=0$ and $\Pi = \delta_{1} \times \cdots \times \delta_{d}$ with each $\delta_{i}$ a discrete series representation, i.e. $\Pi$ is tempered. Therefore, $L(s, \sigma \times \Pi)$ is also the inverse of the normalized numerator of $\gamma(s, \sigma \times \Pi, \psi_F)$ (Definition \ref{Shahidi's L-function}). Consequently, we have the following equality of $L$-functions:
$$L(s,\sigma \times \pi)=L(s, \sigma \times \Pi)$$
for every discrete series representation $\sigma$ of $GL$.

Now it remains to show the uniqueness of the local functorial lift. This follows from the local converse theorem for $GL_{2n}$ (\cite{H0}). If $\Pi_1$ and $\Pi_2$ are local functorial lifts of $\pi$, i.e. irreducible admissible generic representations of $GL$ that satisfy the equality of $\gamma$- and $L$-functions in Theorem \ref{equality}.  By assumption, we have $\gamma(s, \sigma \times \pi, \psi_F) = \gamma(s, \sigma \times \Pi_1, \psi_F)=\gamma(s, \sigma \times \Pi_2, \psi_F)$. Therefore, local converse theorem for $GL_{2n}$ (Corollary in \cite[Section 1.1]{H0}) implies the uniqueness of the local functorial lift.
\end{proof}

\begin{remark}
In the proof of Theorem \ref{F:local2}, we prove the equality of $\gamma$-factors and the equality of $L$-functions in a more general setting. We have the equality of $\gamma$-factors and the equality of $L$-functions in Theorem \ref{F:local2} when $\rho$ is a generic discrete series representations of $GL_m$.
\end{remark}

\begin{remark}
Since the global functorial lift $\underline{\Pi}$ in Theorem \ref{F:global} is either cuspidal or a full induced representation from cuspidals, $\underline{\Pi}$ is generic. Thus all of its local components are as well.
\end{remark}

\begin{definition}\label{F:local}
Let $\pi$ and $\Pi$ be as in Theorem \ref{F:local2}. We shall call this $\Pi$ the local functorial lift of $\pi$.
\end{definition}

\subsection{local transfer image and Langlands parameters: supercuspidal case}
\label{sc:LP}
$ $

We describe the image of the local transfer lift of the irreducible generic supercuspidal representations and construct its Langlands parameters. We first discuss the properties of poles of $L$-functions. The following lemma is a corollary of \cite[Proposition 7.3]{S2}.

\begin{lemma}\label{F:main:lemma}
Let $\pi$ denote an irreducible unitary generic supercuspidal representation of $G_n$ and $\rho$ an irreducible unitary generic supercuspidal representation of $GL_m$. Let,  for $i=1,2$, $L(s, \rho \otimes \pi, r_i)$  be the $L$-functions for $GSpin$ groups from the Langlands-Shahidi method $($see Section \ref{Shahidi}$)$. Then,

(a) $L(s, \rho \otimes \pi, r_i)$ has a pole at $s=0$ if and only if $L(s, \widetilde{\rho \otimes \pi}, r_i)$ has a pole at $s=0$.

(b) $L(s, \rho \otimes \pi, r_i) \sim L(-s, \widetilde{\rho \otimes \pi}, r_i)$. Here $\sim$ means that two $L$-functions are equal up to monomial in $q^{-s}$.

(c) $L(s, \rho \otimes \pi, r_i)$ can have poles only for $Re(s)=0$.

\end{lemma}

\begin{proof}
From Lemma 5.3.2 in \cite{S7}, we have
$$\overline{\gamma(s,\rho \otimes \pi, r_i, \psi_F)}=\gamma(\overline{s}, \widetilde{\rho \otimes \pi}, r_i,\overline{\psi_F}).$$
Since $\pi$ and $\rho$ are tempered, $\overline{L(s,\rho \otimes \pi, r_i)}=L(\overline{s}, \widetilde{\rho \otimes \pi}, r_i)$.
Therefore, $L(s,\rho \otimes \pi, r_i)$ has a pole at $s=0$ if and only if $L(s, \widetilde{\rho \otimes \pi}, r_i)$ has a pole at $s=0$. This proves (a). From Proposition 7.3 in \cite{S2}, $L(s, \rho \otimes \pi, r_i)$ is of the form $\displaystyle\prod\limits_{j} (1- \alpha_j q^{-s})^{-1}$ with $|\alpha_j|=1$. On the other hand, using $\overline{L(s,\rho \otimes \pi, r_i)}=L(\overline{s}, \widetilde{\rho \otimes \pi}, r_i)$ above, we have
$$L(-s, \widetilde{\rho \otimes \pi}, r_i) = \overline{L(-\overline{s}, \rho \otimes \pi, r_i)}= \displaystyle\prod\limits_{j} (1- \overline{\alpha_j} q^{s})^{-1} = \displaystyle\prod\limits_{j} (-\overline{\alpha_j} q^{s})^{-1}(1- \alpha_j q^{-s})^{-1}.$$ Therefore,  $L(-s, \widetilde{\rho \otimes \pi}, r_i)= ( \displaystyle\prod\limits_{j} (-\alpha_j q^{-s}) )L(s, \rho \otimes \pi, r_i)$. This proves (b). (c) is a direct consequence of Proposition 7.3 in \cite{S2} since $L(s, \rho \otimes \pi, r_i)$ is of the form $\displaystyle\prod\limits_{j} (1- \alpha_j q^{-s})^{-1}$ with $|\alpha_j|=1$.
\end{proof}

\begin{theorem}\label{F:main}
Let $\pi$ be an irreducible unitary generic supercuspidal representation of the group $GSpin_{2n+1}$ $($resp. $GSpin_{2n})$ and let $\Pi$ be its local functorial lift in the sense of Theorem \ref{F:local2}. Then $\Pi$ is of the form
$$\Pi \simeq \Pi_{1} \times \cdots \times \Pi_{d}$$
where each $\Pi_{i}$ is an irreducible unitary supercuspidal representation of some $GL_{2n_i}$ such that $L(s, \Pi_{i} \otimes w_{\pi}, \wedge^2 \otimes \mu^{-1})$ $($resp. $L(s, \Pi_{i} \otimes w_{\pi} , Sym^2 \otimes \mu^{-1}))$ has a pole at $s=0$, $\widetilde{\Pi}_{i} \otimes (\omega_{\pi} \circ \operatorname{det}) \cong \Pi_{i}$ and $\Pi_{i} \ncong \Pi_{j}$ for $i \neq j$.

\end{theorem}

\begin{proof}
We first consider the odd case, i.e., $GSpin_{2n+1}$. By Theorem \ref{F:local2} and its proof, we know that the local functorial lift $\Pi$ is tempered and of the form
$$ \Pi = \delta_{1} \times \cdots \times \delta_{d}$$
with each $\delta_{i}$ a discrete series representation. In Theorem \ref{F:local2}, we also prove that for any discrete series representation $\sigma$ of $GL_m$, $L(s, \sigma \times \pi) = L(s, \sigma \times \Pi)$. We apply this equality of twisted $L$-functions with $\sigma = \widetilde{\delta_{i}}$. Then, we have
\begin{equation}\label{Equation 4.2.1}
L(s, \widetilde{\delta_{i}} \times \pi)=L(s, \widetilde{\delta_{i}} \times \Pi)
\end{equation}

We now claim that each $\delta_{i}$ is in fact supercuspidal. We can realize $\delta_{i}$ as the unique irreducible subrepresentation $\delta([\nu^{-\frac{t_i-1}{2}}\rho_i, \nu^{\frac{t_i-1}{2}}$ $ \rho_i])$ of $\nu^{\frac{t_i-1}{2}} \rho_i \times \cdots \times \nu^{-\frac{t_i-1}{2}} \rho_i$ with $\rho_i$ a unitary supercuspidal representation of some $GL$ (\cite{ZB, Z}). Then, our claim becomes to show that $t_i=1$.

Let us first study the $L$-function on the left hand side of the equation (\ref{Equation 4.2.1}). By the multiplicativity of $\gamma$-factors, we have
$$\gamma(s, \widetilde{\delta_{i}} \times \pi, \psi_F) = \displaystyle\prod\limits_{k=0}^{t_i-1} \gamma(s-\frac{t_i-1}{2}+k, \widetilde{\rho_{i}} \times \pi, \psi_F).$$

Since $\pi$ and $\rho_{i}$ are irreducible unitary generic supercuspidal representations, Lemma \ref{F:main:lemma} (b) implies the following:
$$L(1-s, \widetilde{\rho_{i} \times \pi}) \sim L(s-1, \rho_{i} \times \pi)$$
where $\sim$ denotes that two $L$-functions  are equal up to a monomial in $q^{-s}$.

Therefore,
$$\gamma(s, \widetilde{\delta_{i}} \times \pi, \psi_F) \sim \frac{(\displaystyle\prod\limits_{k=0}^{t_i-1} L(s-\frac{t_i-1}{2}+k, \widetilde{\rho_{i}} \times \pi))^{-1}}{(\displaystyle\prod\limits_{k=0}^{t_i-1} L(s-\frac{t_i-1}{2}+k-1, \widetilde{\rho_{i}} \times \pi))^{-1}} =
\frac{(L(s+\frac{t_i-1}{2}, \widetilde{\rho_{i}} \times \pi))^{-1}}{(L(s-\frac{t_i-1}{2}-1, \widetilde{\rho_{i}} \times \pi))^{-1}}.$$

Since $\widetilde{\delta_{i}} \times \pi$ is tempered, $L(s, \widetilde{\delta_{i}} \times \pi)$ is the inverse of the normalized numerator of $\gamma(s, \widetilde{\delta_{i}} \times \pi, \psi_F)$. Therefore, we have

$$L(s, \widetilde{\delta_{i}} \times \pi)=L(s+\frac{t_i-1}{2}, \widetilde{\rho_{i}} \times \pi).$$

Since $\widetilde{\rho_{i}} \times \pi$ is an irreducible unitary generic supercuspidal representation, Lemma \ref{F:main:lemma} (c) implies that $L(s+\frac{t_i-1}{2}, \widetilde{\rho_{i}} \times \pi)$ can have poles only for $Re(s+\frac{t_i-1}{2})=0$. Therefore, $L(s, \widetilde{\delta_{i}} \times \pi)$ can only have poles on the line $Re(s)= \frac{1-t_i}{2}$.

Let us now consider the $L$-functions on the right hand side of the equation (\ref{Equation 4.2.1}). By multiplicativity of $L$-functions (\cite{JPS}, \cite{S2} or Example \ref{Multi:GL}) we have
$$L(s, \widetilde{\delta_{i}} \times \Pi) = \displaystyle\prod\limits_{j=1}^{d} L(s, \widetilde{\delta_{i}} \times \delta_{j}) \ \ \text{and} \ \ L(s, \widetilde{\delta_{i}} \times \delta_{i})= \displaystyle\prod\limits_{k=0}^{t_i-1} L(s+k, \widetilde{\rho_{i}} \times \rho_{i}).$$

Consequently, $L(s, \widetilde{\delta_{i}} \times \Pi)$ has a pole at $s=1-t_i$ since $L(s+t_i -1, \widetilde{\rho_{i}} \times  \rho_{i})$ has a pole at $s=1- t_i$. Therefore, Equation (\ref{Equation 4.2.1}) implies that $L(s, \widetilde{\delta_{i}} \times \pi)$ also has a pole at $s=1-t_i$. Since $L(s, \widetilde{\delta_{i}} \times \pi)$ can have only poles at $\frac{1-t_i}{2}$, $t_i=1$. Therefore, $\delta_{i}=\rho_{i}$ is supercuspidal and we can write
$$\Pi = \rho_{1} \times \cdots \times \rho_{d},$$
with each $\rho_{i}$ supercuspidal.
To see the other properties in the statement of the theorem, we consider the equality
$$L(s, \widetilde{\rho_{i}} \times \pi) = L(s, \widetilde{\rho_{i}} \times \Pi).$$
The right hand side has a pole at $s=0$ of order equal to the number of $j$ such that $\rho_{i} \simeq \rho_{j}$.
But the left hand side can have at most a simple pole at $s=0$, as $\widetilde{\rho_{i}} \times \pi$ is supercuspidal. Hence we see that $\rho_{i} \ncong \rho_{j}$ if $i \neq j$ and $L(s, \widetilde{\rho_{i}} \times \pi)$ has a simple pole at $s=0$.
Lemma \ref{F:main:lemma} (a) implies that $L(s, \rho_{i} \otimes \pi, r_1)$ has a simple pole at $s=0$ since $L(s, \widetilde{\rho_{i}} \otimes \widetilde{\pi}, r_1):=L(s, \widetilde{\rho_{i}} \times \pi)$ has a simple pole at $s=0$. Therefore, Theorem 8.1 of \cite{S2} implies that $\nu^s \rho_{i} \rtimes \pi$ is reducible at $s=1$ and $\widetilde{\rho_{i}} \otimes (\omega_{\pi} \circ \operatorname{det}) \cong \rho_{i}$.




We finally come to the $L$-function condition. Proposition 7.3 and Corollary 7.6 of \cite{S2} imply that the product
$$L(s, \rho_{i} \otimes \pi, r_1)L(2s, \rho_{i} \otimes \pi, r_2)$$
has a simple pole at $s=0$, where $r_2=Sym^2 \otimes \mu^{-1}$.
By previous analysis, this pole is already accounted for by $L(s, \rho_{i} \otimes \pi, r_1)$. We conclude that $L(s, \rho_{i} \otimes \pi, Sym^2 \otimes \mu^{-1})$ has no pole at $s=0$. On the other hand, from Lemma 3.14 and Remark 3.17 of \cite{Ki1} 
we know that
$$L(s, \rho_{i} \times (\rho_{i} \otimes w_{\pi}^{-1}))=  L(s, \rho_{i} \otimes w_{\pi}, \wedge^2 \otimes \mu^{-1})L(s, \rho_{i} \otimes w_{\pi}, Sym^2 \otimes \mu^{-1})$$
where $w_{\pi}$ is a central character of $\pi$.

Since $L(s, \rho_{i} \times (\rho_{i} \otimes w_{\pi}^{-1}))$ has a simple pole at $s=0$, we conclude that $L(s, \rho_{i} \otimes w_{\pi}, \wedge^2 \otimes \mu^{-1})$ has a pole at $s=0$.

The even case, i.e. $GSpin_{2n}$, is similar to the odd case. One has just to change the roles of $Sym^2$ and $\wedge^2$.
\end{proof}

The following theorem shows that the Langlands parameter of the local functorial lift of an irreducible generic supercuspidal representation $\pi$ of $G_n$ factors through the $L$-group of $\textbf{G}_n$.

\begin{theorem}\label{F:Langlands parameter}
Let $\pi$ and $\Pi$ be as in Theorem \ref{F:main} and let $\phi_{\Pi} : W_F' \rightarrow GL_{2n}(\mathbb{C})$ be the local Langlands parameter that is attached to $\Pi$ by the local Langlands correspondence for general linear groups (\cite{HT, H1}). Then $\phi_{\Pi}$ factors through the $L$-group of $\textbf{G}_n$.
\end{theorem}

\begin{proof}
We first consider the odd case, i.e. $GSpin_{2n+1}$ groups.  By Theorem \ref{F:main}, $\Pi \simeq \Pi_{1} \times \cdots \times \Pi_{d}$ where each $\Pi_{i}$ is an irreducible unitary supercuspidal representation of some $GL_{2n_i}$ such that $L(s, \Pi_{i} \otimes w_{\pi}, \wedge^2 \otimes \mu^{-1})$ has a pole at $s=0$, $\widetilde{\Pi}_{i} \otimes (\omega_{\pi} \circ \operatorname{det}) \cong \Pi_{i}$ and $\Pi_{i} \ncong \Pi_{j}$ for $i \neq j$. Let $\phi_{i}=\phi_{\Pi_{i}}: W_{F} \rightarrow  \,^{L}GL_{2n_i}$ be the local Langlands parameter attached to $\Pi_{i}$ through the local Langlands correspondence for general linear groups (\cite{HT, H1}). Then $\Pi$ has the local Langlands parameter $\phi_{\Pi}= \phi_{1} \oplus \phi_{2} \oplus \cdots \oplus \phi_{d}$. Henniart proved that the twisted exterior square $L$-functions are Artin $L$-functions (\cite{H2}). Therefore, the Artin $L$-function $L(s, (\wedge^2 \otimes \mu^{-1}) \circ (\phi_{i} \otimes \phi_{w_{\pi}}) )= L(s, \Pi_{i} \otimes w_{\pi}, \wedge^2 \otimes \mu^{-1})$ also has a pole at $s=0$. This implies that $(\wedge^2 \otimes \mu^{-1}) \circ (\phi_{i} \otimes \phi_{w_{\pi}})$ contains a trivial representation. Therefore, each $\phi_{i}$ factors through $GSp_{2n_i}(\mathbb{C})$ for all $i$. Then, $\phi_{\Pi}$ also factors through $GSp_{2n}(\mathbb{C})$.
The even case, i.e. $GSpin_{2n}$ is similar to the odd case. One has just to change the roles of $Sym^2$ and $\wedge^2$.
\end{proof}

\subsection{The equality of $L$-functions and Langlands parameters: general case}
\label{H}
$\text{ } \text{ } \text{ } $
Denote, for an irreducible admissible generic representation $\rho $ of $GL_k$, by $\phi _{\rho }$ the admissible
homomorphism $W_F\rightarrow GL_k(\Bbb C)$  attached to $\rho $  by the local Langlands correspondence.

The following theorem follows from a special case of \cite[Theorem 3.2]{He2} with help of the local Langlands correspondence for $GL_n$ \cite{HT}:

\begin{theorem}\label{F:reduction} Let $\pi $ be an irreducible admissible generic representation of  $G_n$ and
$\sigma $ be an irreducible admissible generic representation of $GL_m$. Fix a standard Levi subgroup $\textbf{M}= \textbf{GL}_{k_1} \times \cdots \times \textbf{GL}_{k_s} \times \textbf{G}_k$ of $\textbf{G}_n$ and
an irreducible generic supercuspidal representation $\pi_{sc}=\rho_1\otimes\cdots\otimes\rho_r \otimes \tau$
of $M$, such that $\pi$ is a sub-representation of $i_P^{\textbf{G}_n}\pi_{sc}$.

Suppose that there is an admissible homomorphism $\phi _{\tau }:W_F\rightarrow\ ^L\textbf{G}_k$ such that, for every
irreducible supercuspidal representation $\rho $ of every general linear group $GL_l$, $l\geq 1$, one has the equalities
of local Rankin $\gamma $-factors
$$\gamma(s ,\rho \times \tau ,\psi _F)=\gamma (s ,\phi _{\rho } \otimes \phi_{\tau },\psi _F)$$
and the equality of local symmetric square (resp. exterior square) $L$-functions
$$L(s, \rho\times\tau, \wedge^2)=L(s, \phi _{\rho }\otimes\wedge^2(\phi_{\tau })),\qquad \hbox{\it if $k$ is odd,}$$
$$(\hbox{\it resp.}\qquad L(s, \rho\times\tau, Sym ^2)=L(s, Sym ^2(\phi _{\rho }\otimes\phi_{\tau }))\qquad \hbox{\it if $k$ is even.})$$

Then the assumptions of
\cite[Theorem 6.3]{He1} are satisfied which attach to $\pi $ an admissible homomorphism $$\phi _{\pi }:W_F'\rightarrow\ ^L\textbf{G}_n.$$
Moreover, one has
$$\gamma (s ,\sigma\times\pi ,\psi _F)=\gamma (s ,\phi _{\sigma }\otimes\phi_{\pi },\psi _F)$$
and in particular,
$$L(s ,\sigma\times\pi )=L(s ,\phi _{\sigma }\otimes\phi_{\pi }).$$
\end{theorem}

\begin{remark} One could add the equality of the symmetric/exterior square L-functions to the conclusions of the theorem, but we do not need it here. \end{remark}

\begin{proof} Remark first that the equalities for local Rankin $\gamma $-factors and local Rankin $L$-functions associated to supercuspidal representations of $GL_{n_1} \times GL_{n_2} $, $n_1 ,n_2$ any integer $\geq 1$, follow from the local Langlands correspondence \cite{HT}. By this and the equality of the other $L$-functions, the assumptions of  \cite[Theorem 2.3]{He2} are satisfied, which implies that the Langlands parameter $\phi_{\pi }:W_F'\rightarrow\ ^L\textbf{G}_n$ is well defined. The equality of the local Rankin $\gamma $-factors can be shown as in the proof of \cite[Theorem 3.2]{He2}, because in the product formula for the local Rankin  $\gamma $-factors only Rankin  $\gamma $-factors do appear (Example \ref{Multi:GSpin}).
\end{proof}

\begin {theorem}
\label{main}
 For $\pi $ an irreducible admissible generic representation of  $G_n$ and $\sigma $ an irreducible admissible generic representation of $GL_m$, one has the equality of local
$\gamma $-factors
$$\gamma (s ,\sigma\times\pi ,\psi _F)=\gamma (s ,\phi _{\sigma }\otimes\phi_{\pi },\psi _F)$$
and the equality of local $L$-functions
$$L(s ,\sigma\times\pi )=L(s ,\phi _{\sigma }\otimes\phi_{\pi }).$$

\end {theorem}

\begin{proof}
We apply theorem \ref{F:reduction} with the notations therein: it follows from theorem \ref{F:Langlands parameter} that the local Langlands parameter of
the lift $\Pi(\tau )$ of $\tau $ to $GL_{2k}$ defined in theorem \ref{F:local2} factors through the $L$-group of $\textbf{G}_n$. Also by theorem \ref{F:local2}, equality of the local Rankin $\gamma $-factors holds. It has been proved by G. Henniart \cite[Theorem 1.4]{H2} that one has equality for the exterior square (resp. symmetric square) local $L$-functions. Consequently, theorem \ref{F:reduction} can be applied and implies the equality for local Rankin $\gamma -$factors and $L$-functions as stated. \end{proof}

\begin{remark}
One can approach theorem \ref{F:reduction}  also by using the classification of strongly positive representations of $GSpin$ groups \cite{Kim0, Kim, Kim1, Kim2}. The method of the present paper uses Langlands parameters, which are not considered in those papers. Although no specific classification of discrete series representations is apparent in the present paper, this is hidden in the use of the main result of \cite{He3} on which \cite{He1, He2} are based.  \end{remark}

\begin{remark}
In the case of classical groups, the equality of local factors are first proved in \cite{CKSS2}. Our results can be used in the case of classical groups in a completely analogous manner to construct Langlands parameter. This gives a new and simple proof of the equality of $L$-functions in this case.
\end{remark}

The following is the generic Arthur packet conjecture in the case of split $GSpin$ groups (Theorem B in Section \ref{Intro}):
\begin{corollary}
 Let $\psi$ be a local Arthur parameter for split $GSpin$ groups. Let $\Pi(\phi_{\psi})$ be the $L$-packet attached to the Langlands parameter $\phi_{\psi}$ which corresponds to the Arthur parameter $\psi$. Suppose that $\Pi(\phi_{\psi})$ has a generic member. Then it is a tempered $L$-packet.
\end{corollary}

\begin{proof}
This follows from Theorem \ref{main} by a result of Shahidi (\cite[Theorem 5.1]{S4} who proved that, if the equality of $L$-functions through the local Langlands correspondence holds, then the generic Arthur packet conjecture is also true.
\end{proof}

\section{Erratum}
\label{Erratum}
The first author wants to use the opportunity to correct the statement of theorem 4.6 in \cite{He1} (see also proposition 4.9 of the same article), where it is said that in an $L^2$-pair $(s,N)$ in a complex connected reductive group $\mathcal G$ the nilpotent element $N$ is determined by $s$ up to a nonzero constant. Of course, if one denotes $N=\sum _{\alpha }N_{\alpha }$ the decomposition of $N$ relative to the decomposition of the Lie algebra of $\mathcal G$ into root spaces for the action of a maximal torus $\mathcal T$, $s\in \mathcal T$, only the roots $\alpha $ with $N_{\alpha }\ne 0$ are determined by $s$. (It is immediate by the definition of an $L^2$-pair that only these $N_{\alpha }$ can be nonzero, because $\alpha(s)$ must be equal to $q$. From proposition 5.8.5 in \cite{Ca} it follows that all these $N_{\alpha }$ must in fact be nonzero.)

This result is used in \cite{He1} in the proof of the proposition 4.10 which associates to an $L^2$-pair $(s,N)$ a morphism of algebraic groups $\phi :SL_2(\Bbb C)\rightarrow\mathcal G$ and says that this morphism is uniquely determined by $s$ upto equivalence. However, unicity upto equivalence is still true, because the roots $\alpha $ such that $N_{\alpha }\ne 0$ must be linearly independent by the above and, consequently, if $(s,N)$ and $(s,N_1)$ are two $L^2$-pairs, it is always possible to find an element $s_1$ in the torus $\mathcal T$, such that $Ad(s_1)N=N_1$. It follows that the morphism of algebraic groups $\phi _1:SL_2(\Bbb C)\rightarrow\mathcal G$ associated to  $(s,N_1)$ can be obtained from $\phi $ by conjugating by $s_1$, which means that both are equivalent.

The first author thanks P. Schneider to have remarked to him the mistake in the statement of theorem 4.6 in \cite{He1}

\section{Acknowledgment}

The second author wants to express his deepest gratitude to his advisor, F. Shahidi for his constant encouragement and help. The part of this paper (Supercuspidal case) is based on the second author's thesis (\cite{Kim0}) which F. Shahidi suggested as one of the projects. The second author would like to thank Sandeep Varma for carefully reading a paper to help us to improve the presentation of this paper.

\bibliographystyle{amsplain}

\end{document}